\documentclass[reqno,12pt]{amsart}%
\usepackage{amsfonts}
\usepackage{amsmath}
\usepackage{amssymb}

\theoremstyle{plain}
 \newtheorem{thm}{Theorem}[section]
 \newtheorem{prop}{Proposition}[section]
 
 \newtheorem{cor}{Corollary}[section]
\theoremstyle{definition}

\theoremstyle{remark}
 
 \numberwithin{equation}{section}

\def\wc{\overset{d}{=}}
\def\wcL{\overset{{\mathcal{L}}}{=}}
\def\zint{-\!\!\!\!\!\!\int}

\title[Kelly Criterion]{Kelly Criterion: From a Simple Random Walk to L\'{e}vy
Processes}



\author[Lototsky]{Sergey Lototsky} 
\author[Pollok]{Austin Pollok}

\address{ 
Sergey Lototsky \\
Department of Mathematics \\ 
USC   \\ 
Los Angeles, CA 90089\\
USA}
\email{lototsky@usc.edu}

\address{ 
Austin Pollok \\
Department of Mathematics \\ 
USC   \\ 
Los Angeles, CA 90089\\
USA}
\email{pollok@usc.edu}

\begin{document}

{\begin{flushleft}\baselineskip9pt\scriptsize
\today\\
\end{flushleft}}

\begin{abstract}
The original Kelly criterion provides a strategy to maximize the
long-term growth of winnings in a sequence of simple Bernoulli bets with
an edge, that is, when the expected return on each bet is positive.
The objective of this work is to
consider  more general models of returns and the continuous time,
or high frequency, limits of those models.

\end{abstract}

\maketitle
\section{Introduction}
Consider repeatedly engaging in a game of chance where one side has an edge
and seeks to optimize the betting in a way that ensures maximal long-term
growth rate of the overall wealth.  This problem was posed and analyzed
 by John Kelly \cite{Kelly56} at Bell Labs in 1956; the solution was
  implemented and tested in a variety of setting by a
   successful mathematician, gambler, and  hedge fund manager  Ed Thorp  \cite{Thorp06}
 over the period from  the 60's to the early 00's.

As a motivating example, consider betting on a biased coin
toss where the return $r$ is a random variable with distribution
\begin{equation}
\label{return1}
\mathbb{P}(r=1) = p,\ \ \ \mathbb{P}(r=-1)=1- p;
\end{equation}
in what follows, we refer to this as the {\em simple Bernoulli model}.
The condition to have an edge in this setting becomes  $1/2<p\leq 1$ or, equivalently,
\begin{equation}
\label{edge1}
    \mathbb{E}[r] =2p-1>0.
\end{equation}
We plan on being able to make a large sequence of bets on this biased coin, resulting in an iid sequence of returns
$\{r_k\}_{k\geq 1}$  with the same distribution as $r$, and ask how much we should bet so as to maximize
long term wealth,
 given that we are compounding our returns. Assume we are betting
 with a fixed exposure $f$, that is, each bet involves a fixed fraction $f$ of the overall
wealth, and   $f \in [0,1]$. Practically, $f \geq 0$ means {\bf no shorting} and $f \leq 1$ means {\bf no leverage}, which we refer to as the {\bf NS-NL} condition.
Then, starting with the initial amount $W_0$,
  the total wealth at time $n=1,2,3,\ldots$ is the following function of $f$:
\begin{equation*}
    W_n^f = W_0 \prod_{k=1}^n\big(1 + f r_k\big).
\end{equation*}
For the long-term compounder wishing to maximize their long term wealth, a natural and equivalent goal would be  to find the strategy $f=f^*$  maximizing the long-term growth rate
\begin{equation}
\label{rate-dt}
 g_r(f):= \lim_{n\rightarrow \infty}\frac{1}{n} \ln \frac{W_n^f}{W_0}.
 \end{equation}
 By direct computation,
 $$
 g_r(f) =
    \lim_{n\rightarrow \infty}\frac{1}{n}\sum_{k=1}^n \ln(1+fr_k)
     =  \mathbb{E}\ln(1+fr)=p\ln(1+f) + (1-p)\ln(1-f),
$$
where the second equality follows by  the law of large numbers, and therefore, after solving $g_r'(f^*)=0$,
\begin{equation}
\label{optimal1}
 f^* = 2p-1,\ \
 \max_{f \in [0,1]} g_r(f)= g_r(f^*)=p\ln \frac{p}{1-p}+(2-p)\ln(2-2p);
\end{equation}
note that the edge condition \eqref{edge1} ensures that $f^*$ is an admissible strategy and $g_r(f^*)>0$.
For more discussions of this result see \cite{Thorp06}.

Our objective in  this paper is to derive analogues of  \eqref{optimal1}
in the following situations:
\begin{enumerate}
\item the distribution of returns is a more general random variable;
\item the compounding is continuous in time;
\item the compounding is  high frequency, leading to a continuous-time limit.
\end{enumerate}
In particular, we consider several scenarios when
the returns are described by  L\'evy processes, which  addresses
some of Thorp's questions  regarding fat-tailed distributions
in finance \cite{Thorp08}.

In what follows, we write $\xi\wc\eta$ to indicate equality in distribution for
two random variables, and $X\wcL Y$ to indicate equality in law (as function-valued random elements) for two random processes. For $x>0$, $\lfloor x \rfloor$
 denotes the largest integer less than or equal to $x$.
 To simplify the notations, we always assume that $W_0=1$.

\section{Discrete Compounding: General Distribution of Returns}

Assume that the returns on  each bet are independent  random variables
 $r_k,\ k\geq 1,$ with the same distribution as a given random variable $r$,
 and let
 \begin{equation}
 \label{wealth2}
  W_n^f =  \prod_{k=1}^n\big(1 + f r_k\big)
 \end{equation}
  denote the corresponding wealth process. We also keep  the NS-NL
  condition on admissible strategies: $f\in [0,1]$.

  For the wealth process $W^f$ to be well-defined, we need the random
  variable $r$ to have the following properties:
\begin{align}
\label{r1}
&\mathbb{P}(r\geq -1)=1;\\
\label{r2}
&\mathbb{P}(r>0)>0,\ \mathbb{P}(r<0)>0;\\
\label{r3}
&\mathbb{E}|\ln(1+r)|<\infty.
\end{align}

Condition \eqref{r1} quantifies the idea that a loss in a bet should not be
more than $100\%$. Condition \eqref{r2} is basic non-degeneracy: both gains and losses are possible.
Condition \eqref{r3} is a minimal requirement to define the
 long-term growth rate of the wealth process.

The key object in this section will be the function
\begin{equation}
\label{F(f)}
g_r(f)=\mathbb{E}\ln(1+fr).
\end{equation}
In particular, the following result shows that $g_r(f)$ is the long term
growth rate of the wealth process $W^f$.

\begin{prop}
\label{prop:LLN}
If \eqref{r1} and \eqref{r3} hold and $g_r(f)\not=0$,
then, for every $f\in[0,1]$,  the
wealth process $W^f$ has an asymptotic representation
\begin{equation}
\label{eq:LLN-g1}
W^f_n=\exp\Big( n g_r(f)\big(1+\varepsilon_n\big)\Big),
\end{equation}
where
\begin{equation}
\label{asymp0}
\lim_{n\to \infty} \varepsilon_n=0
\end{equation}
with probability one.
\end{prop}

\begin{proof}
By \eqref{wealth2}, we have \eqref{eq:LLN-g1} with
\begin{equation}
\label{err0}
\varepsilon_n=\frac{1}{ng_r(f)}\sum_{k=1}^n\bigg( \ln(1+fr_k)-g_r(f)\bigg),
\end{equation}
and then \eqref{asymp0} follows by \eqref{r3} and the strong law of
large numbers.
\end{proof}

A stronger version of \eqref{r3} leads to a more detailed asymptotic
of $W_n^f$.

\begin{thm}
\label{th:asympt1}
Assume that \eqref{r1} holds and
\begin{equation}
\label{r3-3}
\mathbb{E}|\ln(1+r)|^2<\infty.
\end{equation}
Then
then, for every $f\in[0,1]$,  the
wealth process $W^f$ has an asymptotic representation
\begin{equation}
\label{eq:LLN-g2}
W^f_n=\exp\Big( n g_r(f)+\sqrt{n}\big(\sigma_r(f)\zeta_n+
\epsilon_n\big)\Big),
\end{equation}
where  $\zeta_n,\ n\geq 1, $ are standard Gaussian random variables,
\begin{equation*}
\sigma_r(f)=\Big(\mathbb{E}\big[\ln^2(1+fr)\big]-g_r^2(f)\Big)^{1/2},
\end{equation*}
and
\begin{equation*}
\lim_{n\to \infty}\epsilon_n=0
\end{equation*}
in probability.
\end{thm}

\begin{proof}
With $\varepsilon_n$ from \eqref{err0}, the result follows by the Central Limit Theorem:
$$
ng_r(f)\,\varepsilon_n= \sqrt{n}\left(\frac{1}{\sqrt{n}}
\sum_{k=1}^n\bigg(\ln(1+fr_k)-g_r(f)\bigg)\right)=
\sqrt{n}\big(\sigma_r(f)\zeta_n+
\epsilon_n\big).
$$
\end{proof}
Because the Central Limit Theorem gives convergence in distribution, the random variables
$\zeta_n$ in \eqref{eq:LLN-g2} can indeed depend on $n$.  Additional assumptions about the
distribution of $r$  \cite[Theorem 1]{Zolotarev-AE}
 lead to higher-order asymptotic expansions and a possibility to have
$\lim_{n\to \infty}\epsilon_n=0$ with probability one.

The following properties of the function $g_r$ are immediate
consequences of the definition and the assumptions \eqref{r1}--\eqref{r3}:
\begin{prop}
\label{prop:BasicF}
The function $f\mapsto g_r(f)$ is continuous on the closed interval
$[0,1]$ and infinitely differentiable in $(0,1)$. In particular,
\begin{equation}
\label{DF}
\frac{dg_r}{df}(f)=\mathbb{E}\left[\frac{r}{1+fr}\right],\ \
\frac{d^2g_r}{df^2}(f)=-\mathbb{E}\left[\frac{r^2}{(1+fr)^2}\right]<0.
\end{equation}
\end{prop}

\begin{cor}
\label{cor1}
The function $g_r$ achieves its maximal value on
$[0,1]$ at a point  $f^*\in [0,1]$ and $g_r(f^*)\geq 0$.
If $ g_r(f^*)> 0$, then $f^*$ is unique.
\end{cor}

\begin{proof}
Note that $g_r(0)=0$ and, by \eqref{DF}, the function $g_r$ is strictly concave
(or convex up) on $[0,1]$.
\end{proof}

While concavity of $g_r$ implies that $g_r$ achieves  a unique global maximal
value at a point $f^{**}$, it is possible that the domain of
the function $g_r$ is bigger than
the interval $[0,1]$ and  $f^{**}\notin [0,1]$. A simple way to exclude the possibility $f^{**}<0$ is to consider returns
$r$ that are not bounded from above: $\mathbb{P}(r>c)>0$ for all $c>0$:
in this case, the function $g_r(f)=\mathbb{E}\ln(1+fr)$ is not defined for $f<0$.
Similarly, if $\mathbb{P}(r<-1+\delta)>0$ for all $\delta>0$, then the
function $g_r$ is not defined for $f>1$, excluding the possibility $f^{**}>1.$

 Below are  more general sufficient conditions to ensure that the point
 $f^*\in [0,1]$ from Corollary \ref{cor1} is the point of global maximum of $g_r$:
$f^*=f^{**}$.

\begin{prop}
\label{prop:glob1}
If
\begin{align}
\label{global1-l}
&\lim_{f\to 0+} \mathbb{E}\left[\frac{r}{1+fr}\right]>0 \ \
\ \ \ {\rm and}\\
\label{global1-r}
&\lim_{f\to 1-} \mathbb{E}\left[\frac{r}{1+fr}\right]<0,
\end{align}
then there is a unique $f^*\in (0,1)$ such that
$$
g_r(f)< g_r(f^*)
$$
 for all $f$ in the domain of $g_r$.
 \end{prop}

 \begin{proof}
 Together with the intermediate value theorem, conditions \eqref{global1-l}
 and \eqref{global1-r} imply that there is a unique $f^*\in (0,1)$ such that
 $$
 \frac{dg_r}{df}(f^*)=0.
 $$
 It remains to use strong concavity of $g_r$.
 \end{proof}

 Because $r\geq -1$, the expected value $\mathbb{E}[r]$ is always defined,
 although $\mathbb{E}[r]=+\infty$ is a possibility. Thus, by \eqref{DF},
 condition \eqref{global1-l} is equivalent to the intuitive idea of an edge:
  $$
  \mathbb{E}[r]>0,
  $$
 which, similar to   \eqref{edge1}, guarantees  that $g_r(f)>0$ for some $f\in (0,1)$.
 Condition \eqref{global1-r} can  be written as
  $$
  \mathbb{E}\left[\frac{r}{1+r}\right]<0,
  $$
  with the convention that the left-hand side can be $-\infty$. This condition
  does not appear in the simple Bernoulli model, but
 is necessary in general, to ensure that the edge is not too big and
leveraged gambling  ($f^*>1$) does not lead to an optimal strategy.

 As an example, consider the {\em general Bernoulli model} with
 \begin{equation}
 \label{GenBern}
 \mathbb{P}(r=-a)=1-p,\ \ \mathbb{P}(r=b)=p,\ \ 0<a\leq 1,\ b>0,\ 0<p<1.
 \end{equation}
 The function
 $$
 g_r(f)=p\ln (1+fb) + (1-p)\ln(1-fa)
 $$
 is defined on $(-1/b, 1/a)$,  achieves the global maximum at
 $$
 f^*=\frac{p}{a}-\frac{1-p}{b},
 $$
 and
 $$
 g_r(f^*)=p\ln p +(1-p)\ln (1-p) +\ln\frac{a+b}{a} - (p-1)\ln \frac{b}{a};
 $$
 we know that $g_r(f^*)\geq 0$, even though it is not at all obvious from the
 above expression.

 The NS-NL condition $f^*\in [0,1]$ becomes
 $$
 \frac{a}{a+b}\leq p \leq \min\left( \frac{ab}{a+b}\left(1+\frac{1}{b}\right), 1\right),
 $$
 and it is now easy to come up with a model in which $f^*>1$: for example, take
 $$
 a=0.1, \ b=0.5,\ p=0.5
 $$
 so that $f^*=4$. Given that a gain and a loss  in each bet are equally likely, but the amount of a gain
 is five times as much as that of a loss,
  a large value of $f^*$ is not surprising, although economical and financial implications of this
 type of leveraged betting are potentially very interesting and should be a subject of a
 separate investigation.

Because of the logarithmic function in the definition of $g_r$, the distribution of $r$ can have a rather
heavy right tail and still  satisfy \eqref{r3}. For example, consider
 \begin{equation}
 \label{Ch0}
 r=\eta^2-1,
 \end{equation}
 where $\eta$ has standard Cauchy distribution with probability density function
 $$
 h_{\eta}(x)=\frac{1}{\pi(1+x^2)},\  \  \ -\infty< x<+\infty.
 $$
 Then
 $$
 g_r(f)=\frac{2}{\pi}\int_0^{+\infty} \frac{\ln\big((1-f)+fx^2\big)}{1+x^2}\, dx= 2\ln\big(\sqrt{f}+\sqrt{1-f}\big),
 $$
 where the second equality follows from \cite[(4.295.7)]{Gradshtein-Ryzhyk}.
 As a result, we get a  closed-form answer
 $$
 f^*=\frac{1}{2},\ g_r(f^*)=\ln 2.
 $$

 A general  way to ensure \eqref{r1}--\eqref{r3} is to consider
 \begin{equation}
 \label{expo-model}
 r=e^{\xi}-1
 \end{equation}
 for some random variable $\xi$ such that  $\mathbb{P}(\xi>0)>0,\
 \mathbb{P}(\xi<0)>0$, and $\mathbb{E}|\xi|<\infty$;
  note that \eqref{Ch0} is a particular case, with $\xi=\ln\eta^2$.
 Then  \eqref{global1-l}  and  \eqref{global1-r} become, respectively,
 \begin{align}
\label{global1-l-e}
&\mathbb{E}e^{\xi}>1 \ \
\ \ \ {\rm and}\\
\label{global1-r-e}
&\mathbb{E}e^{-\xi}>1.
\end{align}
 For example, if $\xi$ is normal with mean $\mu\in \mathbb{R}$
 and variance $\sigma^2>0$, then
 $$
 \mathbb{E}e^{\xi}=e^{\mu+(\sigma^2/2)},
  \ \ \mathbb{E}e^{-\xi}=e^{-\mu+(\sigma^2/2)},
  $$
  and \eqref{global1-l-e}, \eqref{global1-r-e} are equivalent to
  \begin{equation}
  \label{log-nrmal}
  -\frac{\sigma^2}{2}<\mu<\frac{\sigma^2}{2},
  \end{equation}
   which, when interpreted in terms of returns,
   can indeed be considered as a ``reasonable'' edge condition: large values of  $|\mu|$
    do create a bias in one direction.
  Note that the corresponding $f^*$ is not available in closed
  form, but can be evaluated numerically.

\section{Continuous Compounding and a Case for L\'{e}vy Processes}
\label{sec:CC}
Continuous time compounding includes discrete compounding  as a particular case and
makes it possible to consider more general types of return processes.
The objective of this section is to show that continuous time compounding
that leads to a non-trivial and non-random  long-term growth rate of the resulting wealth
process effectively forces the return process to have independent increments.
The two main examples of such process are sums of iid random variables from the previous
section and the L\'{e}vy processes.

Writing \eqref{wealth2} as
\begin{equation}
\label{wealth1-dt}
W_{n+1}^f-W_n^f=\big(fW_n^f\big)\,r_{n+1},
\end{equation}
we see that a natural continuous time version of \eqref{wealth1-dt}
is
\begin{equation}
\label{wealth1-ct}
dW_{t}^f=fW_t^fdR_{t}
\end{equation}
for a suitable process $R=R_t,\ t\geq 0$ on a stochastic basis
\begin{equation*}
\mathbb{F}=\Big(\Omega, \mathcal{F},\ \{\mathcal{F}_t\}_{t\geq 1},
\mathbb{P}\Big)
\end{equation*}
satisfying the usual conditions \cite[Definition I.1.1]{Protter}.
We interpret  \eqref{wealth1-ct} as an integral equation
\begin{equation}
\label{wealth1-ct-i}
W_{t}^f=1+f\int_0^tW_s^fdR_{s};
\end{equation}
 recall that  $W_0^f=1$ is the standing assumption.
Then the Bichteler-Dellacherie theorem \cite[Theorem III.22]{Protter}
implies that the process $R$ must be a semi-martingale (a sum of a martingale
and a process of bounded variation) with trajectories that, at every point,
 are continuous from the
right and have limits from the left. Furthermore, if we allow the process $R$
to have discontinuities, then, by
\cite[Theorem II.36]{Protter}, we need to modify
\eqref{wealth1-ct-i} further:
\begin{equation*}
W_{t}^f=1+f\int_0^tW_{s-}^fdR_{s},
\end{equation*}
where
$$
W_{s-}=\lim_{\varepsilon\to 0, \varepsilon>0} W_{s-\varepsilon},
$$
and, assuming $R_0=0$,  the process $W^f$ becomes  the Dol\'{e}ans-Dade exponential
\begin{equation}
\label{DDE-1}
W^f_t=\exp\left(fR_t-\frac{f^2\langle R^c\rangle_t}{2}
\right)\prod_{0<s\leq t}
(1+f\triangle R_s)\,e^{-f\triangle R_s}.
\end{equation}
In \eqref{DDE-1}, $\langle R^c\rangle$ is the quadratic variation process of the
continuous martingale component of $R$ and $\triangle R_s=R_s-R_{s-}$.

A natural analog of \eqref{r1} is
\begin{equation}
\label{r1-ct}
\triangle R_s\geq -1,
\end{equation}
and then \eqref{DDE-1} becomes
\begin{equation}
\label{DDE-2}
W^f_t=\exp\left(fR_t-\frac{f^2\langle R^c\rangle_t}{2}+
\sum_{0<s\leq t}\Big( \ln
(1+f\triangle R_s) -f\triangle R_s\Big)\right).
\end{equation}

To proceed, let us assume that the trajectories of $R$ are continuous:
$\triangle R_s=0$ for all $s$ so that
\begin{equation*}
W^f_t=\exp\left(fR_t-f^2\langle R^c\rangle_t\right).
\end{equation*}

If, similar to \eqref{rate-dt},  we define the
long-term growth rate $g_R(f)$ by
\begin{equation}
\label{gr-ct0}
g_R(f)=\lim_{t\to \infty} \frac{\ln W^f_t}{t},
\end{equation}
then  we need the limits
\begin{equation}
\label{trip0-cont}
\mu:=\lim_{t\to \infty} \frac{R_t}{t},\ \ \
\sigma^2:=\lim_{t\to \infty} \frac{\langle R^c\rangle_t}{t}
\end{equation}
to exist with probability one and with non-random numbers $\mu,\sigma^2.$
 Being a semi-martingale without jumps, the process
$R$ has a representation
\begin{equation}
\label{sm-cont}
R_t=A_t+R^c_t,
\end{equation}
where $A$ is process of bounded variation;
cf. \cite[Theorem II.2.34]{LimitTheoremsforStochasticProcesses}.
Then \eqref{trip0-cont}
imply that, for large $t$,
\begin{equation}
\label{linear0}
A_t\approx \mu t,\ \
\langle R^c\rangle_t\approx \sigma^2 t,
\end{equation}
 that is, a natural
way to achieve \eqref{trip0-cont} is to consider the process $R$ of the form
\begin{equation*}
R_t=\mu t+\sigma\,B_t,
\end{equation*}
where $\sigma>0$ and $B=B_t$ is a standard Brownian motion.
 Then
 \begin{equation}
\label{DDE-3-cont}
W^f_t=\exp\left(f\mu t+f \sigma\,B_t-\frac{f^2\sigma^2 t}{2}\right),
\end{equation}
and we come to the following conclusion:
{\em continuous time compounding with a continuous return process
effectively implies that the wealth process is a geometric Brownian motion}.

 The long-term growth rate
 \eqref{gr-ct0} becomes
 \begin{equation}
 \label{gr-ctL-cont}
 g_R(f)=f\mu-\frac{f^2\sigma^2}{2},
\end{equation}
so that
\begin{equation*}
f^*=\frac{\mu}{\sigma^2},\ \ g_R(f^*)=\frac{\mu^2}{2\sigma^2},
\end{equation*}
and the NS-NL condition is
\begin{equation*}
0<\mu<\sigma^2.
\end{equation*}
Even though these results are not especially sophisticated, we will see
in the next section (Theorem \ref{prop0}) that the process
\eqref{DDE-3-cont} naturally appears as the continuous-time, or high
frequency,  limit of discrete-time compounding for a large class of returns.

On the other hand, if we assume that the process $R$ is purely
discontinuous, with jumps $\triangle R_{k}=r_k$
 at times $s=k\in \{1,2,3,\ldots\}$,
then
$$
R_t=0,\ t\in (0,1),\ R_t=\sum_{k=1}^{\lfloor t \rfloor} r_k,\ t\geq 1,
$$
and  \eqref{DDE-1} becomes \eqref{wealth2}.
Accordingly, we will now investigate the general case \eqref{DDE-1} when the
process $R$  has  both a continuous component and  jumps.
To this end, we use
 \cite[Proposition II.1.16]{LimitTheoremsforStochasticProcesses}
 and introduce the jump measure
 $\mu^R=\mu^R(dx,ds)$ of the process $R$ by putting a point mass at
 every point in space-time where the process $R$ has a jump:
 \begin{equation}
 \label{JumpMeasure}
 \mu^R(dx,ds) =
 \sum_{s > 0}\delta_{(\triangle R_s,s)}(dx, ds);
 \end{equation}
  note that both the time $s$ and size $\triangle R_s$ of the jump can be random.
  In particular, with \eqref{r1-ct} in mind,
\begin{equation}
\label{sum1}
\sum_{0<s\leq t} \Big(\ln
(1+f\triangle R_s) -f\triangle R_s\Big)=
\int_0^t\zint_{-1}^{+\infty}
\big(\ln(1+fx)-fx\big)\mu^R(dx,ds);
\end{equation}
here and below,
\begin{equation}
\label{zint}
\zint_a^b, \ \ \ a<0<b,
\end{equation}
stands for
$$
\int\limits_{(a,0)\bigcup(0,b)}.
$$

By \cite[Proposition II.2.9 and Theorem II.2.34]{LimitTheoremsforStochasticProcesses},
and keeping in mind \eqref{r1-ct}, we get the following generalization of \eqref{sm-cont}:
\begin{equation}
\label{sm-general}
\begin{split}
R_t&=A_t+R^c_t+
\int_0^t\zint_{1}^{+\infty}
x\mu^R(dx,ds)\\
&+
\int_0^t\zint_{-1}^1
x\big(\mu^R(dx,ds)-\nu(dx,s)da_s\big),
\end{split}
\end{equation}
where $a=a_t$ is a predictable non-decreasing process
and $\nu=\nu(dx,t)$  is the non-negative random time-dependent
measure on $(-1,0)\bigcup(0,+\infty)$ with the property
\begin{equation*}
\zint_{-1}^{+\infty}\min(1,x^2)\nu(dx,t)\leq 1
\end{equation*}
for all $t\geq 0$ and $\omega\in \Omega$.
Moreover
\begin{align}
\label{trip-A}
A_t&=\int_0^t \mu_s\, da_s \ {\rm \  for\  some\  predictable \ process} \
\mu=\mu_t,\\
\label{trip-B}
\langle R^c\rangle_t&=\int_0^t \sigma^2_s\,da_s\
{\rm \  for\  some\  predictable \ process} \ \sigma=\sigma_t,
\end{align}
and the process
$$
t \mapsto \int_0^t\zint_{-1}^{+\infty}
h(x)\big(\mu^R(dx,ds)-\nu(dx,s)da_s\big)
$$
is a martingale for every bounded measurable function $h=h(x)$ such that
$\limsup_{x\to 0}|h(x)|/|x|<\infty$.

To proceed, we assume that
$$
\mathbb{E} \int_0^t\zint_{-1}^{+\infty} |\ln(1+x)|\,
\nu(dx,s)\,da_s<\infty,\ t>0,
$$
 which is  a generalization  of condition \eqref{r3}.
Then, by
\cite[Theorem II.1.8]{LimitTheoremsforStochasticProcesses}, the process
$$
t\mapsto \int_0^t\zint_{-1}^{+\infty}
\ln(1+x)\big(\mu^R(dx,ds)-\nu(dx,s)da_s\big)
$$
is a martingale.

 Next, we combine \eqref{DDE-2}, \eqref{sum1}, and \eqref{sm-general},
 and re-arrange the terms so that the logarithm  of the wealth process
 becomes
\begin{equation}
\label{logW-g}
\begin{split}
\ln W_t^f&=fA_t+fR^c_t-\frac{f^2}{2}\langle R^c_t \rangle
-f\int_0^t\zint_{-1}^1x\nu(dx,s)da_s\\
&+ \int_0^t\zint_{-1}^{+\infty} \ln(1+fx)\nu(dx,s)da_s+M^f_t,
\end{split}
\end{equation}
where
$$
M^f_t=\int_0^t\zint_{-1}^{+\infty}
\ln(1+fx)\big(\mu^R(dx,ds)-\nu(dx,s)da_s\big).
$$
In general, for equality \eqref{logW-g} to hold,
 we need to make an additional assumption
\begin{equation}
\label{nu-int1}
\zint_{-1}^{1}x\nu(dx,t)<\infty
\end{equation}
for all $t\geq 0$ and $\omega\in \Omega$.

 In the particular case  \eqref{wealth2},
 \begin{itemize}
 \item $a_s=\lfloor s \rfloor$ is the step function, with unit jumps at positive integers, so that
     $da_s$ is the collection of point masses at positive integers;
     \item $\nu(dx,s)=F^R(dx),$ where
     $F^R$ is the cumulative distribution function of the random variable $r$,
     so that \eqref{nu-int1} holds automatically;
     \item $\mu_t=g_f(r)+\int_{-1}^1 x\,F^R(dx)$,
     $R^c_t=0$, $\sigma_t=0$;
 	\item $M^f_t=\sum_{0<k\leq t}\big(\ln(1+fr_k)-g_r(f)\big)$;
 	\item condition \eqref{LogVar-Levy}  is \eqref{r3-3}.
 \end{itemize}

A natural way to reconcile \eqref{linear0} with \eqref{trip-A},
 \eqref{trip-B} is to take $\mu_t=\mu$, $\sigma_t=\sigma$ for some
 non-random numbers $\mu\in \mathbb{R}$, $\sigma\geq 0$, and  a non-random
 non-decreasing function $a=a_t$ with the property
 \begin{equation}
 \label{limit-a}
 \lim_{t\to+\infty}\frac{a_t}{t}=1.
 \end{equation}
 Then, to have a non-random almost-sure limit
 $$
 \lim_{t\to \infty} \frac{1}{t}\int_0^t\int_{-1}^{+\infty} \varphi(x)
 \nu(dx,s)da_s
 $$
 for a sufficiently rich class of non-random test functions $\varphi$,
 we have to assume that there exists a  non-random non-negative
 measure $F^R=F^R(dx)$ on $(-1,0)\bigcup(0,+\infty)$ such that
 \begin{equation}
 \label{intF1-1}
 \zint_{-1}^{+\infty}\min( |x|,1)\, F^R(dx)<\infty
 \end{equation}
 and, for large $s$,
 $$
 \nu(dx,s)\approx F^R(dx).
 $$
 As a result, if
 \begin{equation}
 \label{LevyMeasure0}
 \nu(dx,s)= F^R(dx)
 \end{equation}
for all $s$, then
\begin{equation}
\label{triple-nr1}
A_t=\mu a_t,\  \langle R^c\rangle_t=\sigma^2a_t,\
 \nu(dx,t)=F^R(dx)a_t
\end{equation}
are all non-random, and
  \cite[Theorem II.4.15]{LimitTheoremsforStochasticProcesses}
implies that  $R$ is a  {\em process with independent increments}.
Furthermore, \eqref{triple-nr1} and the strong law of large numbers for
martingales imply
\begin{equation*}
\mathbb{P}\left(\lim_{t\to \infty} \frac{R^c_t}{t}=0\right)=1;
\end{equation*}
cf. \cite[Corollary 1 to Theorem II.6.10]{LSh-M}.
Similarly, if
\begin{equation}
\label{LogVar-Levy}
\zint_{-1}^{+\infty} \ln^2(1+x)\, F^R(dx)<\infty,
\end{equation}
then  $M^f$ is a square-integrable martingale and
 \begin{equation*}
\mathbb{P}\left( \lim_{t\to \infty} \frac{M^f_t}{t}=0\right)=1.
 \end{equation*}

 Writing
 $$
 \bar{\mu}=\mu-\int_{-1}^{1}x F^R(dx)
$$
the long-term growth rate \eqref{gr-ct0} becomes
 \begin{equation}
 \label{gr-ctL}
 g_R(f)=f\bar{\mu}-\frac{f^2\sigma^2}{2}+
\zint_{-1}^{\infty}
\ln(1+fx) F^R(dx),
\end{equation}
which does include both \eqref{F(f)} and \eqref{gr-ctL-cont} as particular cases.
By direct computation, the function $f\mapsto g_R(f)$ is concave and the
domain of the function contains $[0,1]$.

Similar to Proposition \ref{prop:glob1},  we have the following result.
\begin{thm}
\label{th-LevyRate}
Consider  continuous-time compounding with return process
\begin{equation}
\label{Return-Levy0}
\begin{split}
R_t&=A_t+R_t^c+
\int_0^t\zint_{1}^{+\infty}
x\mu^R(dx,ds)\\
&+
\int_0^t\zint_{-1}^1
x\big(\mu^R(dx,ds)-\nu(dx,s)da_s\big),
\end{split}
\end{equation}
where the random measure $\mu^R$ is from \eqref{JumpMeasure},
and assume that equalities  \eqref{limit-a} and \eqref{triple-nr1} hold.
If $F^R$ satisfies \eqref{intF1-1},  \eqref{LogVar-Levy}, and
\begin{align*}
&\lim_{f\to 0+}
\zint_{-1}^{\infty}
\frac{x}{1+fx}\, F^R(dx)>-\bar{\mu},\\
&\lim_{f\to 1-}
\zint_{-1}^{\infty}
\frac{x}{1+fx}\, F^R(dx)<\sigma^2-\bar{\mu},
\end{align*}
then  the long-term growth rate is given by \eqref{gr-ctL},  and there exists
 a unique $f^*\in (0,1)$ such that
$$
g_R(f)< g_R(f^*)
$$
 for all $f$ in the domain of $g_R$.
\end{thm}

 By the Lebesgue decomposition theorem, the measure corresponding to the
 function $a=a_t$ has a discrete, absolutely continuous, and singular components.
 With \eqref{limit-a} in mind, a natural choice of the discrete component is $a_t=\lfloor t \rfloor$,
 which, as we saw, corresponds to discrete compounding discussed in the previous
 section. A natural choice of the absolutely continuous component is
\begin{equation*}
a_t=t.
\end{equation*}
Then
\begin{equation*}
A_t=\mu t,\  R^c_t=\sigma B_t,\ \nu(dx,t)da_t=F^R(dx)\,dt,
\end{equation*}
where $B$ is a standard Brownian motion.
By \cite[Corollary II.4.19]{LimitTheoremsforStochasticProcesses},
we conclude that the process $R$ has  independent and stationary increments,
that is, {\em $R$ is a L\'{e}vy process}.
In this case, equality \eqref{Return-Levy0} is known as the L\'{e}vy-It\^{o} decomposition
of the process $R$; cf. \cite[Theorem 19.2]{Sato}.

We do not consider  the singular case in this paper
and leave it for future investigation.

\section{Continuous Limit of  Discrete Compounding}

\subsection{A (Simple)  Random Walk Model}
Following the methodology in \cite[Section 7.1]{Thorp06},
 we assume  compounding a sufficiently large  number $n$ of
 bets  in a   time  period $[0,T]$. The returns $r_{n,1}, r_{n,2},\ldots $
 of the bets are
\begin{equation}
\label{return-nk}
 r_{n,k} =
\frac{\mu}{n} + \frac{\sigma}{\sqrt{n}}\,\xi_{n,k}
\end{equation}
for some $\mu>0$, $\sigma>0$ and
independent  identically distributed
random variables $\xi_{n,k},\ k=1,2,\ldots,$ with mean $0$ and variance $1$.
The classical simple random walk corresponds to
$\mathbb{P}(\xi_{n,k}=\pm 1)=1/2$ and can be considered a
{\em high frequency} version of \eqref{return1}.
Similar to \eqref{r1}, we need $r_{n,k}\geq -1$, which, in general,  can only be
 achieved with  {\em uniform boundedness} of $\xi_{n,k}$:
\begin{equation}
\label{eq:bnd1}
|\xi_{n,k}|\leq C_0,
\end{equation}
and then,  with no loss of generality, we assume that
$n$ is large enough so that
\begin{equation}
\label{small-r}
|r_{n,k}|\leq \frac{1}{2}.
\end{equation}

Similar to \eqref{edge1}, a condition to have an edge is
$$
\mathbb{E}[r_{n,k}] = \frac{\mu}{n} > 0,
$$
and, similar to  \eqref{wealth2}, given $n$ bets per unit time period,
 with exposure $f \in [0,1]$ in each bet, we get the following formula for the
 total wealth $W_t^{n,f}$  at time $t\in (0,T]$ assuming $W_0=1$:
\begin{equation}
\label{Wntf}
W_t^{n,f} = \prod_{k=1}^{\lfloor nt \rfloor}\big(1+fr_{n,k}\big);
\end{equation}
 $\lfloor nt \rfloor$ denotes the  largest integer less than $nt$.

Let $B=B_t,\ t\geq 0,$ be a standard Brownian motion on
a stochastic basis $(\Omega, \mathcal{F}, \{\mathcal{F}_t\}_{t\geq 0},\mathbb{P})$ satisfying the usual conditions, and define the process
\begin{equation}
\label{wealth2-cont}
W^f_t= \exp\left(\left(f \mu - \frac{f^2\sigma^2}{2}\right)t + f \sigma B_t\right).
\end{equation}
Note that \eqref{wealth2-cont} is a particular case of \eqref{DDE-3-cont}.

\begin{thm}
\label{prop0}
For every $T>0$ and every $f\in [0,1]$, the sequence of processes
$ \big(W_t^{n,f},\ n\geq 1,\ t\in [0,T]\big)$ converges in law to the process $W^f=W^f_t,\ t\in [0,T]$.
\end{thm}

\begin{proof}
Writing
$$
Y^{n,f}_t= \ln W_t^{n,f},
$$
the objective is to show weak convergence, as $n\to \infty$,  of $Y^{n,f}$ to the process
$$
Y^f_t=\left( f \mu - \frac{f^2\sigma^2}{2}\right)t + f \sigma B_t,\ \  t\in [0,T].
$$
The proof relies on the method of  predictable  characteristics
for semimartingales  from  \cite{LimitTheoremsforStochasticProcesses}.
More specifically, we make suitable    changes in the proof of  Corollary VII.3.11.

By \eqref{wealth2-cont}
$$
Y^{n,f}_t=\sum_{k=1}^{\lfloor nt \rfloor} \ln(1+fr_{n,k}).
$$
Then \eqref{return-nk} and \eqref{eq:bnd1} imply
\begin{equation*}
\mathbb{E}\bigg( Y^{n,f}_t - \mathbb{E}Y^{n,f}_t\bigg)^4
\leq \frac{C_0^4\sigma^4}{n^2}\big(nT+3nT(nT-1)\big)
\leq {3C_0^4\sigma^4T^2},
\end{equation*}
from which uniform integrability of the family
$\{Y^{n,f}_t,\ n\geq 1, \ t\in [0,T]\}$ follows.

Then, by \cite[Theorem VII.3.7]{LimitTheoremsforStochasticProcesses}, it suffices to establish the following:
\begin{align}
\label{mean0}
\lim_{n\to \infty}&
\sup_{t \leq T}
 \left|\lfloor nt \rfloor \mathbb{E}\big[\ln(1+fr_{n,1})\big] - \left(f\mu-\frac{f^2\sigma^2}{2}\right)t \right| =0,\\
 \label{var0}
\lim_{n\to \infty} &
\lfloor nt \rfloor \left(\mathbb{E}
\big(\ln(1+fr_{n,1})\big)^2
-
\bigg(\mathbb{E}\big[\ln(1+fr_{n,1})\big]\bigg)^2 \right)= f^2 \sigma^2 t,\ t\in [0,T],\\
\label{jump0}
\lim_{n\to \infty} &\lfloor nt \rfloor \mathbb{E}
\big[\phi\big(\ln(1+fr_{n,1})\big)\big] = 0,\ t\in [0,T].
\end{align}
Equality \eqref{jump0} must hold for all functions  $\phi=\phi(x),
\ x\in \mathbb{R},$ that are
continuous and bounded on $\mathbb{R}$ and satisfy $\phi(x)=o(x^2),
 \ x\to 0$, that is,
 \begin{equation}
 \label{phi0}
\lim_{x\to 0} \frac{\phi(x)}{x^2}=0.
  \end{equation}

Equalities \eqref{mean0} and \eqref{var0} follow from
$$
r_{n,1}^2=\frac{\sigma^2}{n}\, \xi_{n,1}^2
+ \frac{2\mu\sigma\xi_{n,1}}{n^{3/2}}+\frac{\mu^2}{n^2},
$$
together with \eqref{small-r} and an elementary inequality
$$
\left| \ln(1+x)-x-\frac{x^2}{2} \right| \leq |x|^3,\ |x|\leq \frac{1}{2}.
$$
In particular,
\begin{equation*}
\mathbb{E}
\big[\big(\ln(1+fr_{n,1})\big)^2\big] = \frac{f^2\sigma^2}{n}+o(1/n),
\ n\to +\infty.
\end{equation*}
To establish \eqref{jump0}, note that \eqref{phi0} and \eqref{return-nk}
imply
$$
\phi\big(\ln (1+fr_{n,1})\big) = o(1/n),
\ n\to +\infty.
$$
\end{proof}

Similar to \eqref{rate-dt}, we  define the long-term continuous time growth
rate
\begin{equation*}
g(f)=\lim_{t\to \infty} \frac{1}{t}\ln W^f_t.
\end{equation*}
Then a simple computation show that
$$
g(f)=f\mu  - \frac{f^2\sigma^2}{2},
$$
 and so
 \begin{equation}
 \label{optf-rv}
 f^*=\frac{\mu}{\sigma^2}
 \end{equation}
  achieves the maximal  long-term
 continuous time growth  rate
\begin{equation}
 \label{optf-rv1}
g(f^*)=\frac{\mu^2}{2\sigma^2}.
\end{equation}
The NS-NL condition $f^*\in [0,1]$ holds if
$0\leq \mu\leq \sigma^2$, which, to the order $1/n$,    is consistent with
\eqref{global1-l} and \eqref{global1-r}, when applied to
\eqref{return-nk}:
$$
\mathbb{E}[r_{n,k}]=\frac{\mu}{n},\ \
\mathbb{E}\left[\frac{r_{n,k}}{1+r_{n,k}}\right]
=\frac{\mu-\sigma^2}{n}+o(n^{-1}).
$$

The wealth process \eqref{wealth2-cont} is that of someone who is ``continuously" placing bets, that is,   adjusts the positions instantaneously, and, for large $n$,
 is a good approximation of  high frequency betting \eqref{Wntf}.
In  general, when the returns are given by \eqref{return-nk}, a direct optimization of
\eqref{Wntf} with respect to $f$ will not  lead to a closed-form expression for the corresponding
optimal strategy  $f_n^*$ , but
Theorem \ref{prop0} implies that, for sufficiently large $n$,
 \eqref{optf-rv}  is an approximation of  $f_n^*$
  and \eqref{optf-rv1} is an approximation of the corresponding long-term growth rate.
As an illustration, consider the high-frequency version
 of the simple Bernoulli model \eqref{return1}:
\begin{equation}
\label{return1-hf}
\mathbb{P}\left( r_{n,k}=\frac{\mu}{n}\pm \frac{\sigma}{\sqrt{n}}\right)=\frac{1}{2},
\end{equation}
which, for fixed $n$, is
is a particular case of the general Bernoulli model
\eqref{GenBern} with $p=q=1/2$,
$$
a=\frac{\sigma}{\sqrt{n}}-\frac{\mu}{n},\ b=\frac{\sigma}{\sqrt{n}}+\frac{\mu}{n}.
$$
Then, by direct computation,
$$
f_n^*=\frac{\mu}{\sigma^2-(\mu^2/n)}\to \frac{\mu}{\sigma^2},\ n\to \infty,
$$
and
$$
\lim_{n\to \infty} g_{r_n}(f_n^*)=\frac{\mu^2}{2\sigma^2}.
$$
Even though the analysis of the proof of Theorem \ref{prop0}
shows that the convergence is uniform in $f$ on compact sub-sets of $(0,1)$,
the proof that $\lim_{n\to \infty} f^*_n=f^*$ would require
a version of Theorem \ref{prop0} with $T=+\infty$, which, for now, is not
available.

With natural  modifications, Theorem \ref{prop0}  extends to the setting \eqref{expo-model}.

\begin{thm}
\label{prop0-1}
Assume that
\begin{equation*}
r_{n,k}+1=\exp\left(\frac{b}{n}+\frac{\sigma}{\sqrt{n}}\,\xi_{n,k}\right),
\end{equation*}
where $b \in \mathbb{R}$, $\sigma>0$, and, for each $n\geq 1$, $k\leq n$,
the random variables
$\xi_{n,k}$ are independent and identically distributed, with zero mean, unit variance, and, for every $a>0$,
\begin{equation}
\label{moment1}
 \lim_{n\to\infty}n\,
  \mathbb{E}\big[|\xi_{n,1}|^2I(|\xi_{n,1}|>a\sqrt{n})\big]=0.
\end{equation}
Then the conclusion of Theorem \ref{prop0} holds with
$$
\mu=b+\frac{\sigma^2}{2}.
$$
\end{thm}

\begin{proof}
Even though a formal Taylor expansion suggests
	$$
	r_{n,k}=\frac{\mu}{n} + \frac{\sigma}{\sqrt{n}}\,\xi_{n,k}+o(1/n)
$$
 we cannot apply   Theorem \ref{prop0} directly because the
 random variables $\xi_{n,k}$ are not necessarily uniformly bounded.
 Still, condition \eqref{moment1}
  makes it possible to verify conditions
 \eqref{mean0}--\eqref{jump0}.
\end{proof}

Condition \eqref{moment1} is clearly satisfied when
$\xi_{n,k},\ k=1,2,\ldots,$ are iid standard normal, which corresponds to
\begin{equation}
\label{discr-ret-gauss}
r_{n,k}=\frac{P_{k/n} - P_{(k-1)/n}}{P_{(k-1)/n}}
\end{equation}
and
\begin{equation}
\label{GBM-returns}
P_t=e^{bt+\sigma B_t}.
\end{equation}
Thus, while the exponential model \eqref{expo-model} with log-normal
returns is not solvable in closed form, the high-frequency version leads to the
(approximately) optimal strategy
\begin{equation}
\label{st-line}
f^*=\frac{b}{\sigma^2}+\frac{1}{2},
\end{equation}
and, under \eqref{log-nrmal}, the NS-NL condition holds: $f^*\in (0,1)$.
Numerical experiments  with $\sigma=1$ and $n=10$ show that the
values of the corresponding optimal $f^*_{10}$ are very close to those given
by \eqref{st-line} for all $b\in (-1/2,1/2)$.

Informally, both Theorems \ref{prop0} and \ref{prop0-1} can be considered
as particular cases of the delta method for the Donsker theorem with drift:
if the sequence of processes
$$
t\mapsto \sum_{k=1}^{\lfloor nt \rfloor} \xi_{n,k}
$$
converges, as $n\to \infty$,  to the processes $t\mapsto
bt+\sigma B_t$ and $\varphi=\varphi(x)$ is a suitable
function with $\varphi(0)=0$, then one would expect the  sequence of processes
$$
t\mapsto \sum_{k=1}^{\lfloor nt \rfloor} \varphi\big(\xi_{n,k}\big)
$$
to converge to the process
$$
t\mapsto \left(\varphi'(0)b+\frac{\varphi''(0)\sigma^2}{2}\right)t+|\varphi'(0)|\sigma B_t.
$$

\subsection{Beyond the Log-Normal Limit}
With the results of  Section \ref{sec:CC} in mind,
 we consider the following generalization of \eqref{discr-ret-gauss}:
 \eqref{GBM-returns}:
$$
r_{n,k} = \frac{P_{k/n} - P_{(k-1)/n}}{P_{(k-1)/n}}, \ k=1,2,\ldots,
$$
 where the process $P=P_t,\ t\geq 0$, has the form $P_t=e^{R_t}$, and
$R=R_t$ is a L\'evy process.
In other words,
\begin{equation}
\label{dctr-gen}
r_{n,k}={e}^{R_{k/n}-R_{(k-1)/n}} - 1.
\end{equation}

As in \eqref{Return-Levy0}, the  process $R=R_t$   can be
decomposed into a drift,  diffusion/small jump,
and large jump components according to the L\'{e}vy-It\^{o}
decomposition \cite[Theorem 19.2]{Sato}:
\begin{equation}
\label{Levy-main}
R_t={\mu}t+\sigma\,B_t+
\int_0^t\zint_{-1}^1
x\big(\mu^R(dx,ds)-F^R(dx)ds\big)+
\int_0^t\int_{|x|>1}
x\mu^R(dx,ds);
\end{equation}
we continue to use the notation $-\!\!\!\!\!\int$ first introduced in \eqref{zint}.

Now that the process $R_t$ is exponentiated,
\begin{itemize}
	\item there is no need to assume that $\triangle R_t\geq -1$ ;
	\item the analog of \eqref{LogVar-Levy} becomes
	$\mathbb{E}|R_1|<\infty$.
\end{itemize}

Equality \eqref{Levy-main}  has a natural interpretation
 in terms of financial risks \cite{Thorp03}: the drift represents  the
   edge (``guaranteed'' return), diffusion and small jumps represent
   small  fluctuations of returns, and
    the large jump component represents  (sudden) large  changes in returns.
Similar to \eqref{Wntf}, the corresponding wealth process is
\begin{equation}
\label{wp-100}
W_{t}^{n,f} =  \prod_{k = 1}^{\lfloor nt \rfloor} \big(1 + f r_{n,k}\big).
\end{equation}

We have the following generalization of Theorem \ref{prop0-1}.
\begin{thm}
Consider the family of processes $W^{n,f}=W_t^{n,f}$, $t\in [0,T]$, $n\geq 1$,
$f\in [0,1],$ defined by \eqref{wp-100}. If
 $r_{n,k}$ is given by \eqref{dctr-gen}, with $P_t=e^{R_t}$, and
$R=R_t$ is a L\'{e}vy process with representation \eqref{Levy-main} and $\mathbb{E}|R_1|< \infty$,
  then, for  every $f\in [0,1]$ and $T>0$,
$$
\lim_{n\to \infty}W^{n,f} \wcL W^f
$$
 in $\mathbb{D}((0,T))$,  where
 \begin{equation}
 \label{WP-generalLP}
 \begin{split}
W^f_t &=\exp\left(f R_t + \frac{ f(1-f)\sigma^2}{2}\, t\right.\\
 &+ \left. \int_0^t \zint_{\mathbb{R}} \Big[\ln\big(1+f(e^x - 1)\big) - f x\Big] \mu^R(dx, ds)\right).
 \end{split}
\end{equation}
\end{thm}
\begin{proof}
By \eqref{dctr-gen} and \eqref{wp-100},
\begin{equation*}
\ln W_t^{n,f} =  \sum_{k=1}^{\lfloor nt \rfloor}\ln\bigg( 1 + f \big( {e}^{R_{k/n}-R_{(k-1)/n}} - 1\big)\bigg).
\end{equation*}
\underline{Step 1:}
For $s \in \big(\frac{k-1}{n}, \frac{k}{n}\big]$, let
\begin{equation}
\label{rnks}
r_s^{n,k} = {e}^{R_s - R_{(k-1)/n}} - 1,
\end{equation}
 and apply the  It\^o's formula  \cite[Theorem II.32]{Protter}
 to the process
 $$
 s\mapsto \ln\big(1+f r_s^{n,k}\big),\ \ s \in \Bigg(\frac{k-1}{n}, \frac{k}{n}\Bigg].
 $$
 The result is
\begin{equation*}
\begin{split}
  \ln\big(1+f r_s^{n,k}\big)
&= \int_{\frac{k-1}{n}}^s \frac{f (1 + r_{u-}^{n,k})}{1+f r_{u-}^{n,k}}\,
 dR_u
+ \frac{\sigma^2}{2}\int_{\frac{k-1}{n}}^s \frac{f(1-f)(1+r_{u-}^{n,k})}{\big(1+fr_{u-}^{n,k}\big)^2}\, du \\
& + \int_{\frac{k-1}{n}}^s \zint_{\mathbb{R}}
 \bigg[\ln\big(1-f+f{e}^{x} (r_{u-}^{n,k}+ 1)\big) \\
& - \ln(1+fr_{u-}^{n,k})  - x \frac{f(1+r_{u-}^{n,k})}{1+f r_{u-}^{n,k}}\bigg]\mu^R(dx, du).
\end{split}
\end{equation*}

\underline{Step 2:}
Putting  $s = \frac{k}{n}$ in the above equality and summing over $k$,
we derive the following expression for $ \ln W_t^{n,f}$:
\begin{align}
\notag
\ln W_t^{n,f} &= \sum_{k=1}^{\lfloor nt \rfloor} \bigg(\int_{\frac{k-1}{n}}^{\frac{k}{n}} h^{(1)}_{n,k}(s) \, dR_s
 + \int_{\frac{k-1}{n}}^{\frac{k}{n}}
 h^{(2)}_{n,k}(s) \, ds
 + \int_{\frac{k-1}{n}}^{\frac{k}{n}} \zint_{\mathbb{R}} h^{(3)}_{n,k}(s,x)  \mu^R(dx, du)\bigg)\\
 \label{main-integrals}
&=  \int_0^t H^{(1)}_{n,t}(s) \, dR_s + \int_0^t H^{(2)}_{n,t}(s) \, ds  + \int_0^t\zint_{\mathbb{R}} H^{(3)}_{n,t}(s,x) \mu^R(dx, ds)\,,
\end{align}
where
\begin{equation*}
\begin{split}
\displaystyle
h^{(1)}_{n,k}(s) & = \frac{f(1+r_{s-}^{n,k})}{1+fr_{s-}^{n,k}}\,,\ \
h^{(2)}_{n,k}(s) =
 \frac{\sigma f(1-f)}{2}\frac{1+r_{s-}^{n,k}}{(1+fr_{s-}^{n,k})^2}\,, \\
h^{(3)}_{n,k}(s,x)&= \ln\big(1-f+fe^x(r_{s-}^{n,k}+1) \big) - \ln(1+fr_{s-}^{n,k}) - fx\frac{1+r_{s-}^{n,k}}{1+fr_{s-}^{n,k}}\,;\\
H^{(i)}_{n,t}(s) &= \sum_{k=1}^{\lfloor nt \rfloor} h^{(i)}_{n,k}(s) \mathbf{1}_{(\frac{k-1}{n}, \frac{k}{n}]}(s), \  i = 1, 2;\  \
H^{(3)}_{n,t}(s,x)  = \sum_{k=1}^{\lfloor nt \rfloor} h^{(3)}_{n,k}(s,x)\mathbf{1}_{(\frac{k-1}{n}, \frac{k}{n}]}(s).
\end{split}
\end{equation*}
\underline{Step 3:}
Because
$$
\lim_{n\to \infty,\, k/n\to s} R_{(k-1)/n}=R_{s-},
$$
equality \eqref{rnks} implies
$$
\lim_{n\to +\infty,\, k/n\to s} r^{n,k}_{s-}= 0
$$
  for all $s$.  Consequently, we have the following
convergence in probability:
\begin{align*}
\lim_{n\to +\infty} H^{(1)}_{n,t}(s)=f,\ &
\lim_{n\to +\infty} H^{(2)}_{n,t}(s)=\frac{\sigma^2 f(1-f)}{2},\\
&\lim_{n\to +\infty} H^{(2)}_{n,t}(s,x)=\ln\big(1+f(e^x-1)\big)-fx.
\end{align*}

To pass to the corresponding limits in \eqref{main-integrals}, we need suitable bounds on the functions $H^{(i)}$, $i=1,2,3$.

Using the inequalities
$$
0<\frac{1+y}{1+ay}\leq \frac{1}{a},\ \ 0<\frac{1+y}{(1+ay)^2}
\leq \frac{1}{4a(1-a)}, \ \ \ y>-1,\ a\in (0,1),
$$
we conclude that
$$
0<h^{(1)}_{n,k}(s)\leq 1,\ 0<h^{(2)}_{n,k}(s)\leq \sigma^2,
$$
and therefore
\begin{equation}
\label{ubound1-2}
0<H^{(1)}_{n,t}(s)\leq 1,\ 0<H^{(2)}_{n,t}(s)\leq \sigma^2.
\end{equation}
Similarly, for $f\in (0,1)$ and $y > -1$,
\begin{equation}
\label{H3bnd00}
\left|\ln \frac{1-f+fe^x(y+1)}{1+fy}-fx\frac{1+y}{1+fy}\right|\leq 2\big(|x|\wedge |x|^2\big),
\end{equation}
so that
$$
|h^{(3)}_{n,k}(s,x)|\leq 2\big(|x|\wedge |x|^2\big)
$$
and
\begin{equation}
\label{ubound3}
\big|H^{(3)}_{n,t}(s)\big| \leq 2\big(|x|\wedge |x|^2\big).
\end{equation}
To verify \eqref{H3bnd00},
 fix $f\in (0,1)$ and $y>-1$, and define the function
 $$
 z(x) = \ln \frac{1-f+fe^x(y+1)}{1+fy},\ x\in \mathbb{R}.
 $$
By direct computation,
\begin{equation*}
\begin{split}
    z(0) &= 0, \\
    z'(x) &= \frac{fe^x(y+1)}{1-f + fe^x(y+1)}=1-\frac{1-f}{1-f + fe^x(y+1)},\\
    z'(0) &= \frac{f(y+1)}{1+fy},
    \end{split}
    \end{equation*}
    so that, using the Taylor formula,
    \begin{equation}
    \label{H3bnd}
  \ln \frac{1-f+fe^x(y+1)}{1+fy}-fx\frac{1+y}{1+fy}
  =z(x)-z(0)-xz'(0)=\int_0^x(x-u)z''(u)du.
  \end{equation}
  It remains to notice that
  $$
    0\leq z'(x)\leq  1, \ \ 0\leq z''(x)\leq 1,
    $$
 and then  \eqref{H3bnd00} follows from  \eqref{H3bnd}.

With \eqref{ubound1-2} and \eqref{ubound3} in mind,  the dominated convergence theorem
 \cite[Theorem IV.32]{Protter} makes it possible to pass to the limit
 in probability in \eqref{main-integrals}; the convergence in the
 space $\mathbb{D}$ then follows from the general
results of \cite[Section IX.5.12]{LimitTheoremsforStochasticProcesses}.
\end{proof}

The following is a representation of the
long-term growth rate of the limiting wealth process $W^f$.

\begin{thm}
\label{th:gr-rate-LP-gen}
Let  $R=R_t$ be a L\'{e}vy process with representation \eqref{Levy-main}.
If $\mathbb{E}|R_1| < \infty$, then the process $W^f=W^f_t$
defined in \eqref{WP-generalLP}
satisfies
\begin{equation}
\label{GR-gen-LP}
\begin{split}
\lim_{t\to +\infty} \frac{\ln W_t^f}{t}&=f \bigg(\mu + \int_{|x|>1} x F^R(dx)\bigg)
+ \frac{ f(1-f) \sigma^2}{2}\\
 &+ \int_{\mathbb{R}}\big[\ln\big(1 + f ({e}^x - 1) \big) - f x \big] F^R(dx).
 \end{split}
\end{equation}
\end{thm}
\begin{proof}
By \eqref{WP-generalLP},
$$
\frac{\ln W_t^f}{t}=
f \frac{R_t}{t} + \frac{ f(1-f)\sigma^2}{2}+ \frac{1}{t}\int_0^t \int_{\mathbb{R}} \Big[\ln\big(1+f(e^x - 1)\big) - f x\Big] \mu^R(dx, ds).
$$
It remains to apply the law of large numbers
 for L\'evy processes \cite[Theorem 36.5]{Sato}.
\end{proof}

If, in addition, we assume that
$$
\zint_{-1}^1 |x|F^R(dx)<\infty,
$$
that is, the small-jump component of $R$ has bounded variation, then, after a change of
variables and  re-arrangement of terms, \eqref{GR-gen-LP} becomes  \eqref{gr-ctL}.
On the other hand, equality \eqref{gr-ctL} is derived for a wider class of return processes
that includes L\'{e}vy processes as a particular case.

Similar to Proposition \ref{prop:glob1},  we also have the following result.
\begin{thm}
\label{th-LevyRate1}
In the setting of Theorem \ref{th:gr-rate-LP-gen}, denote the
right-hand side of \eqref{GR-gen-LP} by $g_R(f)$ and assume that
\begin{align*}
&\lim_{f\to 0+}
\int_{\mathbb{R}}
\left(\frac{e^x - 1}{1+f(e^x - 1)} - x\right)\, F^R(dx)>-\bigg(\mu
+ \frac{\sigma^2}{2} + \int_{|x|>1} x F^R(dx)\bigg),\\
&\lim_{f\to 1-}
\int_{\mathbb{R}}
\left(\frac{e^x - 1}{1+f(e^x - 1)} - x\right)\, F^R(dx)<
-\bigg(\mu +  \int_{|x|>1} x F^R(dx)\bigg),
\end{align*}
Then  there exists
 a unique $f^*\in (0,1)$ such that
$$
g_R(f)< g_R(f^*)
$$
 for all $f$ in the domain of $g_R$.
\end{thm}

\section{Continuous Limit of  Random Discrete Compounding}

The objective of this section is to analyze high frequency limits for betting {\em in business time}. In other words, the number of bets is not known a priori, so that a natural model of the
corresponding  wealth process is
\begin{equation}
\label{weath-mgen}
W_{t}^{n,f}=\prod_{k=1}^{\lfloor \Lambda_{n,t}\rfloor}
(1+fr_{n,k})
\end{equation}
where, for each $n$, the process $t\mapsto \Lambda_{n,t}$ is
a subordinator, that is,
a non-decreasing L\'{e}vy process, independent of all $r_{n,k}$.

To study \eqref{weath-mgen}, we will follow the methodology in
\cite{KZZ}, where convergence of processes is derived after {\em assuming}
 a suitable convergence of the random variables.
 The main result in this connection is as follows.

 \begin{thm}
 \label{th:Levy-gen0}
 Consider the following objects:
 \begin{itemize}
 \item random variables $X_{n,k},\ n,k\geq 1$ such that
 $\{X_{n,k},\ k\geq 1\}$ are iid for each $n$,
  with mean zero and, for some $\beta\in [0,1]$,  $m_n:=\Big(\mathbb{E}|X_{n,1}|^{\beta}\Big)^{1/\beta}<\infty$;
 \item random processes $\Lambda_n=\Lambda_{n,t}$, $n\geq 1,\ t\geq 0,$
 such that, for each $n$, $\Lambda_n$ is a subordinator
 independent of $\{X_{n,k},\ k\geq 1\}$ with the properties
  $\Lambda_{n,0}=0$, and
 for some numbers $0<\delta,\delta_1\leq 1$ and $C_n>0$,
 $\Big(\mathbb{E}\Lambda_{n,t}^{\delta}\Big)^{1/\delta}
 \leq C_n t^{\delta_1/\delta}$.
 \end{itemize}
 Assume that there exist infinitely divisible random variables $Y$ and $U$
 such that
 $$
 \lim_{n\to \infty} \sum_{k=1}^n X_{n,k} \wc \bar{Y},\
 \lim_{n\to \infty} \frac{\Lambda_{n,1}}{n} \wc \bar{U}.
 $$

 If
 \begin{equation}
 \label{dens}
 \sup_{n} \Big(C_n m_n^{\beta}\Big)<\infty,
 \end{equation}
 then, as $n\to \infty$,  the sequence of processes
 $$
 t\mapsto \sum_{k=1}^{\lfloor \Lambda_{n,t}\rfloor} X_{n,k},\ t\in [0,T],
 $$
 converges, in the Skorokhod topology, to the process $Z=Z_t$
 such that  $Z_t=Y_{U_t}$, where $Y$ and $U$ are independent
L\'{e}vy processes satisfying $Y_1\wc\bar{Y}$ and $U_1\wc\bar{U}$.
\end{thm}

The proof is a word-for-word repetition of the arguments leading to
\cite[Theorem 1]{KZZ}:
 the result of \cite{GnF}, together with the assumptions of the
  theorem, implies
$$
\lim_{n\to \infty} \sum_{k=1}^{\lfloor \Lambda_{n,1}\rfloor} X_{n,k}\wc Z_1,
$$
and therefore the convergence of finite-dimensional distributions for
 the corresponding processes; together
with condition \eqref{dens}, this  implies the convergence in the Skorokhod space. Because we deal exclusively with L\'{e}vy processes, it is possible to
avoid the heavy machinery from \cite{LimitTheoremsforStochasticProcesses}.

We now consider the wealth process \eqref{weath-mgen}
and apply Theorem \ref{th:Levy-gen0}
with
$$
X_{n,k}=
 \ln (1+fr_{n,k})-\mathbb{E}\ln (1+fr_{n,k}).
 $$

On the one hand,  convergence to infinitely
divisible distributions other than normal is a very diverse area,
with a variety of conditions and conclusions;
cf. \cite[Chapter XVII, Section 5]{FellerII} or a summary in
\cite[Section 16.2]{Klenke}. On the other hand, optimal strategy
\eqref{optf-rv} seems to persist.

For example, assume that
the returns $r_{n,k}$ are as in \eqref{return-nk}, and
let $\Lambda_{n,t}=S_{n^{\alpha}t}$, where $\alpha\in (0,1]$ and
$S=S_t$ is the L\'{e}vy process such that $S_1$ has the $\alpha$-stable
distribution with both scale and skewness parameters equal to $1$.
Recall that an $\alpha$-stable L\'{e}vy process $L^{\alpha}=L^{\alpha}_t$
satisfies the following equality in distribution (as processes):
\begin{equation}
\label{scale-L}
L^{\alpha}_{\gamma t}\wcL \gamma^{1/\alpha}L^{\alpha}_t, \ \gamma>0.
\end{equation}

Then
$$
\Lambda_{n,t}\wcL nS_{t}
$$
 and, in the notations of Theorem \ref{th:Levy-gen0},
$\bar{Y}$ is normal with mean zero and variance $\sigma^2$.
 Keeping in mind that
 $$
 \mathbb{E}\ln (1+fr_{n,k}) = \mathbb{E}\ln (1+fr_{n,1})=
 \Big(f\mu-\frac{f^2\sigma^2}{2}\Big)n^{-1}+o(n^{-1}),
 $$
 we repeat the arguments from \cite[Example 1]{KZZ} to conclude that
 $$
 \lim_{n\to \infty} \ln W^{n,f}_t \wcL \Big(f\mu-\frac{f^2\sigma^2}{2}\Big)S_t
 + Z_t,
 $$
 where $Z_1$ has symmetric $2\alpha$-stable distribution.
 By \eqref{scale-L},
 $$
 S_t\wc t^{1/\alpha}S_1,\ \lim_{t\to +\infty} t^{-1/\alpha}Z_t \wc
 \lim_{t\to +\infty} t^{-1/(2\alpha)}Z_1 \wc 0,
 $$
 and the ``natural''
 long term growth rate becomes
 $$
 \lim_{t\to \infty}t^{-1/\alpha}\Big(\lim_{n\to \infty} \ln W^{n,f}_t\Big) \wc
 \Big(f\mu-\frac{f^2\sigma^2}{2}\Big)S_1,
 $$
 which is random, but, for each realization of $S$, is still maximized by
 $f^*$ from \eqref{optf-rv}. Therefore, if the time with which we compound our wealth is random, then the growth rate is also random as we don't know when we will stop compounding, yet it is still maximized by a deterministic fraction. Note that, for the purpose of this
 computation, the (stochastic) dependence between the processes
 $S$ and $Z$ is not important.

\section{Conclusions And Further Directions}

The NS-NL condition $f^*\in [0,1]$ can fail in many situations.
Even in the simple Bernoulli model, if $p<1/2$, then the short
position $f^*=2p-1$ achieves positive  long-time wealth growth:
$$
g_r(f^*)=p\ln \frac{p}{1-p}+(2-p)\ln(2-2p)=\ln 2 +p\ln p+(1-p)\ln(1-p)>0.
$$
Note that $-p\ln p-(1-p)\ln(1-p)$ is the Shannon entropy
of the Bernoulli distribution, and
the largest value of the entropy is $\ln 2$, corresponding to $p=1/2$.
When the edge is too big (cf. \eqref{GenBern}),
then $f^*>1$, that is, leveraged gambling
 leads to  bigger  long-time wealth growth than any NS-NL strategy.
The economical and financial implications of $f^*\notin [0,1]$ are
beyond the scope of our investigation and must be studied in
a broader context of risk tolerance: even when $f^*\in (0,1)$, a certain
fraction of $f^*$ can be a smarter strategy, cf. \cite[Section 7.3]{Thorp06}.

A related observation, to be further studied in the future, is that high-frequency
betting can lead to a more aggressive strategy than the ``low
frequency'' counterpart. For example, comparing \eqref{return1} and
\eqref{return1-hf}, we see that $\mu=2p-1$ and $\sigma^2=4p(1-p) < 1$
when $p\not=1/2$.
As a result, by \eqref{optf-rv},
the optimal strategy for \eqref{return1-hf} with large $n$ is
$f^*\approx (2p-1)/(4p(1-p))> 2p-1$; recall that $f^*=2p-1$ is the optimal
strategy for the simple Bernoulli model \eqref{return1}. On the other
hand, numerical simulations suggest that, in the log-normal model
\eqref{discr-ret-gauss}, \eqref{GBM-returns}, high-frequency compounding
does not always lead to larger $f^*$.

Other problems warranting further investigation include
\begin{enumerate}
\item A dynamic strategy $f=f(t)$ with a predictable process $f$;
\item A portfolio of bets, with a vector of strategies
  $\mathbf{f}=(f_1,\ldots, f_N)$.
  \end{enumerate}



\bibliographystyle{amsplain}

\end{document}